\documentclass[11pt]{article}
\usepackage{graphicx,float,hyperref, indentfirst} 
\usepackage{amsmath,amsthm,amssymb,amsfonts,geometry,mathtools,xcolor, enumerate}
\usepackage{subfigure}
\usepackage{tikz}
\theoremstyle{plain}
\newtheorem{theorem}{Theorem}[section]
\newtheorem{lemma}[theorem]{Lemma}
\newtheorem{claim}[theorem]{Claim}

\newtheorem{corollary}[theorem]{Corollary}

\newtheorem{remark}[theorem]{Remark}
\newtheorem{definition}[theorem]{Definition}

\newcommand{\Bern}{\text{Bern}}

\newcommand{\Z}{\mathbb{Z}}
\newcommand{\C}{\mathbb{C}}

\newcommand{\E}{\mathbb{E}}
\newcommand{\Var}{\text{Var}}
\newcommand{\Cov}{\text{Cov}}
\newcommand{\mcal}[1]{\mathcal{#1}}

\title{On subgraphs with degrees of prescribed residues in the random graph}
\author{Asaf Ferber \thanks{Department of Mathematics, University of California, Irvine.
Email: \href{mailto:asaff@uci.edu} {\nolinkurl{asaff@uci.edu}}.
Research supported in part by NSF Awards DMS-1954395 and DMS-1953799.}\and Liam Hardiman \thanks{Department of Mathematics, University of California, Irvine.
Email: \href{mailto:lhardima@uci.edu} {\nolinkurl{lhardima@uci.edu}}.}\and Michael Krivelevich\thanks{School of Mathematical Sciences, Tel Aviv University, Tel Aviv, Israel. Email: \href{mailto: krivelev@tauex.tau.ac.il} {\nolinkurl{krivelev@tauex.tau.ac.il}.}
Research supported in part by USA--Israel BSF grant 2018267 and by ISF grant 1261/17.}}
\date{\today}

\begin{document}

\maketitle

\begin{abstract}
    We show that with high probability the random graph $G_{n, 1/2}$ has an induced subgraph of linear size, all of whose degrees are congruent to $r\pmod q$ for any fixed $r$ and $q\geq 2$.
    More generally, the same is true for any fixed distribution of degrees modulo $q$.
    Finally, we show that with high probability we can partition the vertices of $G_{n, 1/2}$ into $q+1$ parts of nearly equal size, each of which induces a subgraph all of whose degrees are congruent to $r\pmod q$.
    Our results resolve affirmatively a conjecture of Scott, who addressed the case $q=2$.
\end{abstract}

\section{Introduction}
In his comprehensive problem book \cite{lovasz1993} (see Ex. 5.17) Lov\'asz states an unpublished but well known result of Gallai:  every graph admits a vertex partition into two sets, each inducing a subgraph with all degrees even.
This result guarantees the existence of (at least) two things: a large induced subgraph (at least one of the parts contains at least half of the vertices), and a vertex partition into subgraphs, with all degrees congruent to zero modulo 2.
We can ask if these guarantees still hold if we change the residue and/or modulus of the desired subgraph degrees. Define the following function to address the existence of a large induced subgraph of a given graph.
For a graph $G$ and integers $r\geq 0$ and $q\geq 2$ we let
$f(G, r, q)$ be the maximum order of an induced subgraph of $G$ with all degrees congruent to $r\pmod q$ (we set $f(G, r, q)=0$ if such a subgraph does not exist).
In particular, it follows from Gallai's result that $f(G,0,2)\geq |V(G)|/2$ for every graph $G$.

Arguably  the most studied case is that of odd subgraphs, i.e., when $r=1,q=2$.
Since an odd graph cannot contain isolated vertices, we restrict our attention to graphs with minimum degree at least one.
A folklore conjecture in graph theory asserts that in this case there exists a constant $c>0$ such that every such graph has an odd subgraph of order at least $c|V(G)|$.
Following some partial results by Caro \cite{caro1994induced} and by Scott \cite{scott1992large,scott2001odd}, this conjecture was recently proved in \cite{FK} with $c=1/10000$.

It is a challenging problem to find a linear lower bound for other values of $r,q$ (perhaps imposing some necessary conditions on $G$). This suggests studying the asymptotic behavior $f(G,r,q)/|V(G)|$ when $G$ is a random graph.
Recall that for $p\in [0,1]$, $G_{n,p}$ is the random variable that outputs a graph on $n$ vertices, where each potential (unordered) pair of vertices is included as an edge with probability $p$ independently. In \cite{scott1992large}, Scott showed that with high probability $f(G_{n, 1/2}, 1, 2) \approx cn$ where $c\approx 0.7729$ and asked for extensions to other values of $r,q$. 
We generalize Scott's result in the random setting as follows:

\begin{theorem}\label{thm: one part}
    Let $q \geq 2$ and let $0\leq r<q$ be an integer.
    If $n$ is a large integer and $k(n,q)>0$ is the greatest integer such that $\binom{n}{k}q^{-k}\geq 1$, then with high probability, $|f(G_{n, 1/2}, r, q) - k| = O(\log^{10}n)$.
\end{theorem}

In fact, with small modifications to the proof of the above theorem, we can further strengthen it to degree sequences modulo $q$. That is, let $0\leq \alpha_0,\ldots,\alpha_{q-1}\leq 1$ be such that $\alpha_0+\cdots+\alpha_{q-1}=1$, and let $\boldsymbol{\alpha}:=(\alpha_0,\ldots,\alpha_{q-1})$.
We define $f(G,\boldsymbol{\alpha},q)$ to be the largest $k$ for which $G$ contains an induced subgraph on $k$ vertices where, for each $0\leq i\leq q-1$, the number of vertices of degree $i \pmod q$ is either $\lfloor \alpha_ik\rfloor$ or  $\lceil \alpha_i k\rceil$. 
Notice that if $\alpha_r=1$ and $\alpha_i=0$ for all $i\neq r$, then $f(G,\boldsymbol{\alpha},q)$ is just $f(G,r,q)$.


\begin{theorem}\label{thm: one part distrubution}
    Let $q\geq 2$ and let $\boldsymbol{\alpha} = (\alpha_0, \ldots, \alpha_{q-1})\in [0,1]^{q}$ be such that $\alpha_0 + \cdots + \alpha_{q-1} = 1$.
    If $n$ is a large integer and $k(n, q, \alpha)>0$ is the greatest integer such that
    \[
    \binom{n}{k}\binom{k}{k_0, \ldots, k_{q-1}}q^{-k}\geq 1
    \]
    for all choices of $k_i\in \{\lceil \alpha_ik\rceil, \lfloor \alpha_ik\rfloor\}$ that satisfy $k_0 + \cdots + k_{q-1} = k$,
    then with high probability,
    $|f(G_{n, 1/2},\boldsymbol{\alpha},q)-k|=o(n)$.
\end{theorem}


Even though Theorem \ref{thm: one part distrubution} implies Theorem \ref{thm: one part}, for the convenience of the reader we will prove the slightly simpler Theorem  \ref{thm: one part} as well. 

Now we address the second guarantee in Gallai's result -- the existence of a partition into induced subgraphs under some degree constraints.
Given a graph $G$, we say that a partition $V(G) = V_1 \cup \cdots \cup V_k$ is an $(r,q)$-{\bf partition} if all degrees in $G[V_i]$ are congruent to $r\pmod{q}$ for every $1\le i\le k$.
Now define
\[
p(G, r, q) = \min\{k: G\text{ has an $(r,q)$-partition with $k$ parts}\}.
\]
Under this notation, Gallai proved that $p(G, 0, 2) \leq  2$ and in \cite{scott2001odd}, Scott showed that $p(G, 1, 2)$ is finite if and only if every connected component of $G$ has an even order.
In the same paper, Scott also treated random graphs and showed that for $n$ even, with high probability $p(G_{n, 1/2}, 1, n)\leq 3$. He conjectured that for every $r,q$ there exists a constant $\tilde c_q$ so that with high probability $p(G_{n, 1/2}, r, q)\leq \tilde c_q$.
We resolve this conjecture in the affirmative with the following theorem.
\begin{theorem}\label{thm: packing}
    Fix an integer $q\geq 2$ and let $0\leq r<q$ be an integer.
     Then with high probability
     we have $p(G_{n, 1/2},r,q)\leq q+1$. 
\end{theorem}
In light of the recent work of Balister, Powierski, Scott and Tan \cite{scott2021}, this bound is the best possible.
Specifically, they show that if $X_n$ is the number of distinct partitions of $V(G_{n, 1/2})$ into $V_1, \ldots, V_q$, where all degrees in $G[V_i]$ are $r_i\pmod q$, then $X_n$ is asymptotically Poisson in distribution for $q >2 $.
For $q=2$, they obtain a more complicated, but still explicit distribution. In both cases $\mathbb P[X_n=0]$ is bounded away from zero.

We will use the second moment method in order to prove our theorems. The main difference between our approach and Scott's original approach is that when working with modulus $q>2$, the arguments are more involved and requite the aid of discrete Fourier analysis. Even though we do not believe that all of the Fourier-type lemmas appearing in this paper are new, we included full proofs as they seem of an independent interest (for example, for an application of these lemmas to the singularity problem of random symmetric Bernoulli matrices, the reader is referred to \cite{ferber2020singularity}).

\paragraph{Notation}

Our graph theoretical notation is quite standard. In particular, for a graph $G=(V,E)$ and $U\subseteq V$, we use $G[U]$ to denote the subgraph induced by the set $U$. 

As for linear algebraic notation, vectors will appear in boldface and their coordinates will not, e.g. $\boldsymbol{v} = (v_1, \ldots, v_m)$.
We write $\Z_q$ for the set of integers modulo  $q\geq 2$, and for any $d\geq 1$, $\boldsymbol{1}_d$ denotes the vector $(1, \ldots, 1)$ in $\Z_q^d$.
For any fixed integer $q$, we define the function $e_q:\Z_q\to \C$ by $e_q(x) = e^{2\pi i x/q}$.
Finally, $\Bern(p)$ denotes a Bernoulli random variable with success probability $p$.

\section{Auxiliary results}

In this section we state and prove some auxiliary results, to be used in the proofs of our main theorems.  

\subsection{Chernoff's bounds}

We will make use of the following tail estimates for binomial random variables.

\begin{lemma}[Chernoff's inequality (see \cite{as2015})]
    \label{Chernoff}
    Let $X \sim \operatorname{Bin} (n, p)$ and let
    $\mu = \mathbb{E}[X]$. Then
    \begin{itemize}
        \item $\mathbb P[X < (1 - a)\mu] < e^{-a^2\mu/2}$ for every $a > 0$;
        \item $\mathbb P[X > (1 + a)\mu] < e^{-a^2\mu/3}$ for every $0 < a < 3/2$.
    \end{itemize}
\end{lemma}
\begin{remark}\label{CheHyper} The conclusions of Chernoff's inequality remain the same
    when $X$ has the hypergeometric distribution (see \cite{jlr2011}, Theorem~2.10).
\end{remark}





\subsection{Key technical lemmas}
In each of the following lemmas, $q$ denotes a fixed positive integer at least $2$, all equivalences are modulo $q$, and we write $\boldsymbol u \equiv \boldsymbol v \pmod q$ for vectors $\boldsymbol u$ and $\boldsymbol v$ when $u_i\equiv v_i\pmod q$ for all $i$. We believe that not all (if any) of the following lemmas are new, but since we could not find a convenient reference in the literature, we include complete proofs for all of them. We state the lemmas first, deferring their proofs to the next subsection.

The first lemma states that the distribution of the sum of many $\Bern(1/2)$ random variables modulo $q$ is asymptotically uniform.

\begin{lemma}\label{lem: one row}
    Let $n$ be a positive integer and let $\xi_1, \ldots, \xi_n$ be iid $\Bern(1/2)$ random variables.
    Then for a fixed $a\in \Z_q$,
    \[
    \mathbb P[\xi_1 + \cdots + \xi_n \equiv a] = \frac{1}{q}\left(1 + e^{-\Omega(n)}\right).
    \]
\end{lemma}

As an immediate corollary, we have the following result for random matrices with iid $\Bern(1/2)$ entries.

\begin{corollary}\label{lem: non-symmetric}
    Let $M$ be a random $s\times t$ matrix whose entries are iid $\Bern(1/2)$ random variables.
    Then for any $\boldsymbol{v}\in \Z_q^s$,
    \[
    \mathbb P[M\boldsymbol{1}_t \equiv \boldsymbol{v}] = \frac{1}{q^s} + p_{s,t}(\boldsymbol v),
    \]
    where $p_{s, t}(\boldsymbol v) =  \frac{1}{q^s}\left[\left(1 + e^{-\Omega(t)}\right)^s-1\right]$.
\end{corollary}
Note that, while the error $p_{s,t}$ does indeed depend on $\boldsymbol v$ (or else the distribution of $M\boldsymbol 1_t$ would be exactly uniform), our bound on its magnitude is independent of $\boldsymbol v$.
In the case where $s \leq n$ and $\omega(\log n)=t \leq n$, as will be our primary use case, this error is $\frac{1}{q^s}\cdot o(1)$.

The second lemma states that a similar result still holds for \emph{symmetric} Bernoulli random matrices.
If $M$ is the adjacency matrix of $G\sim G(m, 1/2)$, a random symmetric $0/1$ $m\times m$ matrix (with zero diagonal), then the $j$-th entry of $M\boldsymbol{1}_m$ is the degree of the $j$-th vertex.
Observe that if $q$ is even, since the sum of the degrees in a graph is even, we have that $M \boldsymbol{1}_m$ is always some vector $\boldsymbol{v}\in \Z_q^m$ for which $\sum_i v_i$ is even modulo $q$.
Since there are exactly $q^m/2$ such vectors, we obtain slightly different distributions for $M\boldsymbol{1}_m \pmod{q}$ for even and odd $q$. 

\begin{lemma}\label{lem: symmetric}
Let $M$ be a random $m\times m$ symmetric matrix whose diagonal is zero and whose entries above the diagonal are iid $\Bern(1/2)$ random variables. Fix $\boldsymbol{v}\in \mathbb{Z}_q^m$. Then, 
\[\mathbb P[M\boldsymbol{1}_m\equiv \boldsymbol{v}]=\begin{cases} \frac{1}{q^m}+p_m(\boldsymbol v)\,, & \textrm{ if }q \textrm{ is odd}\\
\frac{2}{q^m}+p_m(\boldsymbol v)\,, & \textrm{ if }q \textrm{ is even and }\sum_i v_i\text{ is even} \\
0\,, & \textrm{ if } q \textrm{ is even and }\sum_{i}v_i\text{ is odd}, \end{cases}
\]
where $p_m(\boldsymbol{v})=\frac{1}{q^m}\cdot e^{-\Omega(m)}.$
\end{lemma}
 
The third lemma is a uniformity result about the joint distribution of $\boldsymbol{1}_s^TM$ and $M\boldsymbol{1}_t$ for a random $s\times t$ matrix $M$ with iid $\Bern(1/2)$ entries.
\begin{lemma}\label{lem: both ways}

Let $s,t$ be sufficiently large integers such that $\omega(\log t)=s\leq t$.
If $M$ is an $s\times t$ random matrix with iid $\Bern(1/2)$ entries and $\boldsymbol{u}\in \mathbb{Z}_q^s$ and $\boldsymbol{v}\in \mathbb{Z}_q^t$ are such that $\sum u_i\equiv\sum v_j \pmod q$, then
\[\mathbb P[\boldsymbol{1}_s^TM\equiv \boldsymbol{v}^T \textrm{ and }M\boldsymbol{1}_t\equiv \boldsymbol{u}]=\frac{1+o(1)}{q^{s+t-1}}.\]
\end{lemma}

Observe that if $\boldsymbol{u}\in \mathbb{Z}_q^s$ and $\boldsymbol{v}\in \mathbb{Z}_q^t$ are such that $\sum u_i\not\equiv \sum v_j \pmod q$, then we clearly cannot find a $0/1$ matrix $M$ for which $\boldsymbol{1}_s^TM\equiv\boldsymbol{v} \textrm{ and }M\boldsymbol{1}_t\equiv\boldsymbol{u}$.
Moreover, since there are exactly $q^{s+t-1}$ pairs $(\boldsymbol{u},\boldsymbol{v})$ with $\sum u_i\equiv\sum_jv_j$, we see that for $M$ distributed as above, the pair $(\boldsymbol{1}_s^TM,M\boldsymbol{1}_t)$ is approximately uniformly distributed among all ``feasible'' values.


\subsection{Proofs of key lemmas}
We shall prove the above lemmas using some elementary discrete Fourier analysis (for a thorough introduction, the reader is referred to \cite{stein2011fourier}).
In what follows, $\delta^{(m)}:\Z_q^m\to \{0, 1\}$ is given by:
\[
\delta^{(m)}(\boldsymbol{x}) = \begin{cases}
1\,, &\text{if }\boldsymbol{x} \equiv \boldsymbol{0}\\
0\,,&\text{if }\boldsymbol{x}\not\equiv\boldsymbol{0}.
\end{cases}
\]

\begin{proof}[Proof of Lemma \ref{lem: one row}]
    First note that the claim is clearly true for $q=2$ since a sum of iid $\Bern(1/2)$ random variables is even or odd with equal probability.
    In any case, we write
    \[
    \mathbb P[\xi_1 + \cdots +\xi_n \equiv a] = \E[\delta^{(1)}(\xi_1 + \cdots + \xi_n - a)],
    \]
    and then expand $\delta^{(1)}$ into its Fourier series.
    \[
        \E[\delta^{(1)}(\xi_1 + \cdots + \xi_n - a)] = \frac{1}{q}\sum_{\ell \in \Z_q}\E\big[e_q\big(\ell(\xi_1 + \cdots + \xi_n-a)\big)\big],
    \]
    which, by independence, becomes
    \[
    \frac{1}{q}\sum_{\ell\in \Z_q}e_q(-\ell a)\prod_{j=1}^n\E[e_q(\ell \xi_j)] = \frac{1}{q}\sum_{\ell \in \Z_q}e_q(-\ell a)\left(\frac{1+e_q(\ell)}{2}\right)^n.
    \]
    We isolate the $\ell \equiv 0$ term, apply the triangle inequality and set $p_n(a)$ to be the difference:
    \[
        p_n(a) := \mathbb P[\xi_1 + \cdots + \xi_n \equiv a] - \frac{1}{q},
    \]
    which we then estimate.
    \[
        |p_n(a)| \leq \frac{1}{q}\sum_{\ell = 1}^{q-1}|\cos(\pi \ell/q)|^n.
    \]
    Note that when $q=2$, the right-hand side above is simply zero.
    It is easy to verify that for $1\le\ell\le q-1$, we have: $|\cos(\pi\ell/q)|\leq e^{-2/q^2}$.
    Using this, we obtain
    \begin{align*}
        |p_n(a)| \leq \frac{1}{q}\sum_{\ell = 1}^{q-1}e^{-2n/q^2} = \frac{q-1}{q}e^{-2n/q^2}.
    \end{align*}
This completes the proof. 
\end{proof}

\begin{proof}[Proof of Lemma \ref{lem: symmetric}]
    Like in the proof of Lemma \ref{lem: non-symmetric}, we write the probability of interest as the expectation of a delta function, and then expand it into its Fourier series.
    \begin{equation}\label{eq: fourier}
        \mathbb P[M\boldsymbol{1}_m\equiv\boldsymbol{v}]=\frac{1}{q^m}\sum_{\boldsymbol{\ell}\in \mathbb{Z}_q^m}\mathbb{E}\left[e_q\left(\boldsymbol{\ell}^TM\boldsymbol{1}_m\right)\right]e_q\left(-\boldsymbol{\ell}^T\boldsymbol{v}\right).
    \end{equation}
    
    Now, letting $M_{jk}$ be the entries of the matrix $M$, by the fact that $M_{jk}=M_{kj}$ for all $j$ and $k$, the right-hand side of \eqref{eq: fourier} equals 
    \[
    \frac{1}{q^m}\sum_{\boldsymbol{\ell}\in \mathbb{Z}_q^m}\mathbb{E}\left[e_q\left(\sum_{j<k}(\ell_j+\ell_k)M_{jk}\right)\right]e_q\left(-\boldsymbol{\ell}^T\boldsymbol{v}\right),
    \]
    which by independence becomes
    \begin{equation}\label{probability}
    \frac{1}{q^m}\sum_{\boldsymbol{\ell}\in \mathbb{Z}_q^m}e_q\left(-\boldsymbol{\ell}^T\boldsymbol{v}\right)\prod_{1\leq j<k\leq m}\left(\frac{1 + e_q(\ell_j+\ell_k)}{2}\right).
    \end{equation}
    Let us first consider the case where $q$ is odd. The product in the above expression is 1 if and only if $\boldsymbol{\ell} \equiv \boldsymbol{0}$, in which case we isolate this term, set $p_m(\boldsymbol{v})$ to be the difference, and estimate:
    \[
    |p_m(\boldsymbol{v})| = \left|\mathbb P[M\boldsymbol{1}_m\equiv v]-\frac{1}{q^m}\right|\leq \frac{1}{q^m}\sum_{\boldsymbol{\ell}\in \mathbb{Z}_q^m\setminus\{\boldsymbol{0}\}}\prod_{1\leq j<k\leq m}\left|\cos\left(\frac{\pi}{q} (\ell_j+\ell_k)\right)\right|.
    \]
    
    For each vector $\boldsymbol{\ell}\in \mathbb{Z}_q^m$, we let $a(\boldsymbol{\ell})$ be the number of pairs $1\leq j<k\leq m$ for which $\ell_j+\ell_k\not \equiv0\pmod q$. Next, we again apply the bound $|\cos(\pi r / q)| \leq e^{-2/q^2}$ for $r \neq 0 \pmod q$ to obtain
    \[|p_m(\boldsymbol{v})|\leq \frac{1}{q^m}\sum_{\boldsymbol{\ell}\in \mathbb{Z}_q^m\setminus \{\boldsymbol{0}\}}e^{-2\cdot a(\boldsymbol{\ell})/q^2}.\]
    
    Now, let $L_s$ and $L_{\geq s}$ be the set of all vectors $\boldsymbol{\ell}\in \mathbb{Z}_q^m$ with support of size exactly $s$, and those with support at least $s$, respectively.
    We split the above sum into one over vectors with small support and those with large support.
    \[
    |p_m(\boldsymbol{v})| \leq \frac{1}{q^m}\left(\sum_{1\leq s<m/2}|L_s|e^{-2s(m-s)/q^2} + \sum_{m/2\leq s \leq m}|L_s|\cdot e^{-s(s-2)/2q^2}\right)\,.
    \]
    The first sum comes from pairing the nonzero entries of $\boldsymbol{\ell}\neq \boldsymbol{0}$ with its zero entries, giving a bound of $a(\boldsymbol{\ell}) \geq s(m-s)$.
    For the second sum, start by considering the auxiliary graph on the nonzero entries of $\boldsymbol{\ell}\neq \boldsymbol{0}$ where we connect two vertices if they sum to zero modulo $q$.
    Since $q$ is odd, this graph has no loops and is a vertex disjoint union of complete bipartite graphs, with parts corresponding to residue classes among the entries of $\boldsymbol{\ell}$.
    The maximum number of edges in such a graph is $s^2/4$, which is achieved when half of the support is equal to $r$ and the rest is equal to $-r$ for some residue $r$.
    The number of pairs of entries in the support of $\boldsymbol{\ell}$ that sum to a nonzero residue modulo $q$ is then at least $\binom{s}{2}-s^2/4 = s(s-2)/4$.
    We then estimate $|L_s|$ and crudely bound $|L_{\geq m/2}|$ by $q^m$ to obtain an upper bound for $|p_m(\boldsymbol{v})|$.
    \[
    \begin{split}
    |p_m(\boldsymbol{v})| &\leq \frac{1}{q^m}\left(\sum_{1\leq s<m/2}\binom{m}{s}q^se^{-sm/q^2}+|L_{\geq m/2}|e^{-m^2/16q^2}\right)\\
    &= \frac{1}{q^m}\cdot e^{-\Omega(m)}.
    \end{split}
    \]
    
    When $q$ is even, the entries of $M\boldsymbol{1}_m$ sum to an even residue modulo $q$, so $\mathbb P[M\boldsymbol{1}_m \equiv \boldsymbol{v}] = 0$ when $\boldsymbol{v}$ does not satisfy this condition.
    Assume then that $\sum_j v_j$ is an even residue modulo $q$.
    The product in (\ref{probability}) is 1 if and only if $\boldsymbol{\ell}$ is  $\boldsymbol{0}$ or $\frac{q}{2}\cdot \boldsymbol{1}$, and since $e^{-2\pi i \boldsymbol{\ell}^T\boldsymbol{v}/q} = 1$ for these values of $\boldsymbol{\ell}$, these terms combine to give us a leading term of $\frac{2}{q^m}$. We isolate these terms and bound the difference.
    \[
    \begin{split}
    |p_m(\boldsymbol{v})| &= \left|\mathbb P[M\boldsymbol{1}_m=v]-\frac{2}{q^m}\right|\\
    &\leq \frac{1}{q^m}\sum_{\boldsymbol{\ell}\in \mathbb{Z}_q^m\setminus\{\boldsymbol{0}, \frac{q}{2}\cdot \boldsymbol{1}\}}\prod_{1\leq j<k\leq m}\left|\cos\left(\frac{\pi}{q} (\ell_j+\ell_k)\right)\right|\\
    &\leq \frac{1}{q^m}\sum_{\boldsymbol{\ell}\in \mathbb{Z}_q^m\setminus \{\boldsymbol{0}, \frac{q}{2}\cdot \boldsymbol{1}\}}e^{-2\cdot a(\boldsymbol{\ell})/q^2}\,.
    \end{split}
    \]
    Like before, we let $L_s$ denote the set of vectors in $\Z_q^m$ with support of size exactly $s$, and now we let $L_{s, t}$ denote the set of vectors in $\Z_q^m$ with support $s$ and exactly $t$ entries equal to $q/2$.
    We again split the above sum according to the size of the support of $\boldsymbol{\ell}$ and then further by the number of entries equal to $q/2$.
    \[
        |p_m(\boldsymbol{v})| \leq \frac{1}{q^m}\left(\sum_{s=1}^{3m/4}|L_s|e^{-2s(m-s)/q^2} + \sum_{s=3m/4}^m\sum_{t=0}^{m/4}|L_{s, t}|e^{-(s-t)(s-t-2)/2q^2} + \sum_{s=3m/4}^{m}\sum_{t=m/4}^{\min(s, m-1)}|L_{s,t}| e^{-2t(m-t)}\right).
    \]
    The first sum comes from pairing up the zero and nonzero entries of $\boldsymbol{\ell}\in \Z_q^m \setminus\{\boldsymbol{0}, \frac{q}{2}\cdot \boldsymbol{1}\}$.
    For the second sum, we consider the graph whose vertex set is the nonzero, non-$q/2$ entries of $\boldsymbol{\ell}$ and apply the same argument from the odd $q$ case.
    In the third sum, we pair the entries equal to $q/2$ with those that are not.\\
    
    The first and second sums are $e^{-\Omega(m)}$ by the same argument used for odd $q$.
    To bound the third sum, note that $\ell_i + \ell_j \equiv 0\pmod{q}$ if and only if $(\ell_i+q/2)+(\ell_j+q/2)\equiv 0\pmod{q}$.
    Therefore, by adding $(q/2)\cdot \boldsymbol{1}$ to every vector considered in the third sum, we obtain a set of vectors with support at most $3m/4$.
    The third sum is then at most a subsum of the first, so this upper bound for $|p_m(\boldsymbol{v})|$ is at most $e^{-\Omega(m)}/q^m$.
\end{proof}

\begin{proof}[Proof of Lemma \ref{lem: both ways}]

    We use the same Fourier-based approach we employed to prove Lemmas \ref{lem: non-symmetric} and \ref{lem: symmetric}. Set $p(\boldsymbol{u}, \boldsymbol{v}) = \mathbb P[\boldsymbol{1}_s^TM\equiv\boldsymbol{v},\ M\boldsymbol{1}_t\equiv\boldsymbol{u}]$. Then we have
   \[ 
    \begin{split}
        p(\boldsymbol{u}, \boldsymbol{v}) &= \E[\delta^{(t)}_0(\boldsymbol{1}_s^TM-\boldsymbol{v}^T)\cdot \delta^{(s)}_0(M\boldsymbol{1}_t-\boldsymbol{u})]\\
        &= \frac{1}{q^{s+t}}\sum_{(\boldsymbol{\ell}, \boldsymbol{m})\in \Z_q^s \times \Z_q^t}\prod_{\substack{1\leq j\leq s\\ 1\leq k\leq t}}\E\big[e_q\big(M_{jk}\left(\ell_j + m_k\right)\big)\big] \cdot e_q\left(-\left(\boldsymbol{v^T}\boldsymbol{m} + \boldsymbol{\ell}^T\boldsymbol{u}\right)\right).
    \end{split}
    \]
    
    
    
    
    Note that for all pairs $(\boldsymbol{\ell},\boldsymbol{m})$ for which $\ell_j+m_k \equiv 0\pmod q$ for all $j,k$, we have that the product in the above expression equals $1$.
    Moreover, it is easy to see that $\ell_j+m_k \equiv 0 \pmod q$ for all $j,k$ if and only if there exists some $r\in \mathbb{Z}_q$ for which $\boldsymbol{m}\equiv r\cdot \boldsymbol{1}_t$ and $\boldsymbol{\ell}\equiv -r\cdot \boldsymbol{1}_s$.
    Since there are exactly $q$ such pairs $(\boldsymbol{\ell},\boldsymbol{m})$, by letting $\mathcal Z$ be the set of all pairs $(\boldsymbol{\ell},\boldsymbol{m})\in \mathbb{Z}_q^s\times \mathbb{Z}_q^t$ which are not of this form, we have
    \[
    p(\boldsymbol{u}, \boldsymbol{v}) = \frac{1}{q^{s+t-1}} + \frac{1}{q^{s+t}}\sum_{(\boldsymbol{\ell}, \boldsymbol{m})\in \mcal{Z}}\prod_{\substack{1\leq j\leq s\\ 1\leq k\leq t}}\E\big[e_q\big(M_{jk}\left(\ell_j + m_k\right)\big)\big] \cdot e_q\left(-\left(\boldsymbol{v^T}\boldsymbol{m} + \boldsymbol{\ell}^T\boldsymbol{u}\right)\right)
    \]
    
    
    We define the error $q_{s, t}(\boldsymbol u, \boldsymbol v)$ to be the difference between the left-hand side and the first term on the right-hand side above.
    Applying the exponential bound for the cosine gives
    
    
    
    \[|q_{s, t}(\boldsymbol u, \boldsymbol v)| \leq \frac{1}{q^{t+s}}\sum_z |N(z)|e^{-2z/q^2},\]
    where $N(z)$ is the set of all pairs $(\boldsymbol{\ell}, \boldsymbol{m})\in \mathcal Z$ for which the number of non-zero residues among $\ell_i+m_j$, where $1\leq i\leq s$ and $1\leq j\leq t$, is exactly $z$.
    Note that if $z\geq z_0:=\frac{3}{2}q^2(s+t)\log q$, then  $e^{-2z/q^2}\leq 1/q^{3(s+t)}$ and therefore, even if we use the crude bound $\sum_{z\geq z_0}|N(z)|\leq q^{s+t}$, we obtain
    \begin{equation}\label{large z}
    \sum_{z\geq z_0}|N(z)| e^{-2z/q^2}\leq \frac{1}{q^{2(s+t)}}=e^{-\Omega(s+t)}.
    \end{equation}
    It is thus enough to prove that
    \begin{equation}\label{small z}
    \sum_{0<z< z_0}|N(z)|e^{-2z/q^2}=o(1).
    \end{equation}


    To this end, we fix $0<z<z_0$ and investigate when $(\boldsymbol{\ell}, \boldsymbol{m})\in N(z)$.
    Given a positive integer $d$ and a vector $\boldsymbol{w}\in \mathbb{Z}_q^d$, we define the $r$-th \emph{level set} of $\boldsymbol{w}$, $L_r(\boldsymbol{w})\subseteq [d]$ for some $r\in \Z_q$, by
    \[
    L_r(\boldsymbol{w}) = \{i: w_i \equiv r\pmod q\},
    \]
    and we write $L_{\neq r}$ for the set of all other indices.
    Let us first observe that for $(\boldsymbol{\ell},\boldsymbol{m})\in \mathcal Z$, if there exists some residue $r$ for which
    $|L_r(\boldsymbol{\ell})|\cdot|L_{\neq -r}(\boldsymbol{m})|> z$, then we have that $(\boldsymbol{\ell},\boldsymbol{m})\notin N(z)$.
    
    Suppose that each level set of $\boldsymbol{m}\in \Z_q^t$ has size at most $t/2$.
    For any $\boldsymbol{\ell}\in \Z_q^s$, choose $r\in \Z_q$ such that $|L_r(\boldsymbol{\ell})|\geq s/q$.
    We have $|L_{\neq -r} (\boldsymbol{m})|\geq t/2$, which, for large $s$ and $t$, implies that
    \[|L_r(\boldsymbol{\ell})|\cdot |L_{\neq -r}(\boldsymbol{m})|\geq st/2q>z_0>z\]
    and therefore $(\boldsymbol{\ell},\boldsymbol{m})\notin N(z).$
    
    In order for $(\boldsymbol{\ell}, \boldsymbol{m})$ to lie in $N(z)$ it must then be the case that $\boldsymbol{m}$ has a (unique) level set of size $a > t/2$, say $L_r(\boldsymbol{m})$.
    Observe that we cannot have $|L_{\neq -r}(\boldsymbol{\ell})|> z/a$, or else
    \[|L_{\neq -r}(\boldsymbol{\ell})|\cdot |L_r(\boldsymbol{m})|>z.\]
    Therefore, from now on we assume that  $|L_{\neq -r}(\boldsymbol{\ell})|\leq z/a$ which is equivalent to $|L_{-r}(\boldsymbol{\ell})|\geq s-z/a$.
    Next, if $|L_{\neq r}(\boldsymbol{m})|\geq 2z/s$ then
    \[
    |L_{-r}(\boldsymbol{\ell})|\cdot |L_{\neq r}(\boldsymbol{m})|\geq (s-z/a)\cdot 2z/s = 2z(1-z/as) \geq 2z(1-z/(st/2)),
    \]
    which is larger than $z$ for large $s$ and $t$ and we once again have that $(\boldsymbol{\ell},\boldsymbol{m})\notin N(z)$.
    We then assume that $|L_{ r}(\boldsymbol{m})|> t-2z/s$.
    Since $(\boldsymbol{\ell},\boldsymbol{m})\in \mathcal Z$, then $|L_{\neq -r}(
    \boldsymbol{\ell})|>0$ or $|L_{\neq r}(\boldsymbol{m})|>0$, and therefore the number of non-zero residues among $\ell_i+k_j$ is at least
    \[
        |L_{-r}(\boldsymbol{\ell})|\cdot |L_{\neq r}(\boldsymbol{m})|+|L_{\neq -r}(\boldsymbol{\ell})|\cdot |L_{ r}(\boldsymbol{m})|\geq \min\{s-z/a,t-2z/s\},
    \]
    which is at least $s/2$ for large $s$ and $t$.
    In particular, we see that $N(z)$ is empty for all $0<z\leq s/2$.
    Therefore, we may assume that $s/2<z<z_0$.
    
    All in all, we have
    \[
    \begin{split}
        |N(z)|&\leq \binom{s}{z/a}q^{z/a+1}\binom{t}{2z/s}q^{2z/s}\\
        &\leq q\cdot (qt)^{z/a+2z/s}\\
        &\leq q\cdot e^{\frac{4z}{s}\cdot \log(q\cdot t)}.
    \end{split}
    \]
 
    Indeed, we choose a residue $r$, at most $z/a$ elements of $\boldsymbol{\ell}$ to fill out $L_{\neq -r}(\boldsymbol{\ell})$ and at most $2z/s$ elements of $\boldsymbol{m}$ for $L_{\neq r}(\boldsymbol{m})$.
    The last inequality follows from the fact that $a\geq t/2$ and $t\geq s$.
    Finally, we have
    \begin{align*}
        \sum_{0<z< z_0}|N(z)|e^{-2z/q^2} & = \sum_{z = s/2}^{z_0}|N(z)|e^{-2z/q^2}\\
        &\leq \sum_{z = s/2}^{z_0}\exp\left[ z\left(\frac{4\log(qt)}{s} - \frac{2}{q^2}\right)\right].
    \end{align*}
    Now, since $s = \omega(\log t)$, the argument of the exponential is negative for all $z$ when $s,t$ are sufficiently large, in which case we have
    \begin{align*}
        \sum_{z\leq z_0}|N(z)|e^{-2z/q^2} &\leq z_0 \cdot \exp\left[\frac{s}{2}\left(\frac{4\log(qt)}{s} - \frac{2}{q^2}\right)\right].
    \end{align*}
    Since $z_0 = \frac{3}{2}q^2(s+t)\log q \leq 3q^2 t\log q$ and $s = \omega(\log t)$, the above quantity is $o(1)$.
    We have then established (\ref{small z}), which combined with (\ref{large z}) gives the desired conclusion.
\end{proof}

\section{Large induced subgraphs}
Here we prove Theorem \ref{thm: one part} with the second moment method and then make slight modifications to this proof to arrive at Theorem \ref{thm: one part distrubution}.

\subsection{Proof of Theorem \ref{thm: one part}}
The strategy is to first show that we expect $G:=G_{n,1/2}$ to have many induced subgraphs of the type specified by Theorem \ref{thm: one part}.
We will then show that this number of good subgraphs is asymptotically concentrated around its mean and then apply Chebyshev's inequality to conclude that $G$ has a good subgraph with high probability.

\begin{proof}[Proof of Theorem \ref{thm: one part}]
    For simplicity, we will assume throughout the proof that $q$ is odd and point out where special considerations need to be made for even $q$.
    We may further assume that $q\geq 3$ since the (slightly simpler) case of $q = 2$ was handled by Scott in \cite{scott1992large}.
    
    Fix  a positive integer $k$ to be determined later, and let $X_k$ be the number of $k$-vertex induced subgraphs of $G:= G_{n, 1/2}$, all of whose degrees are $r \pmod q$, henceforth known as ``good'' subgraphs.
    Such subgraphs correspond to $k\times k$ principal submatrices $B$ of the adjacency matrix of $G$ satisfying $B\boldsymbol{1}_k \equiv r\boldsymbol{1}_k \pmod{q}$ (all vector congruences are entrywise modulo $q$, and we omit ``mod $q$'' when this condition is clear from context).
    By Lemma \ref{lem: symmetric}, the expected number of good subgraphs is
    \begin{equation}\label{expectation}
    \mathbb{E}[X_k] = \binom{n}{k}\left(\frac{1}{q^k}+\frac{o(1)}{q^k}\right)
    \end{equation}
    when $q$ is odd, and twice this when $q$ is even and $k$ is such that $G$ could possibly admit such good subgraphs (we may always take $k$ to be even in this case, for instance).
    Now choose $k'$ to be the greatest integer such that $g(k)\coloneqq \binom{n}{k}\frac{1}{q^{k}}$ is at least $1$.
    Since $g(k) \geq (\frac{n}{qk})^k$, we see that $g(k) \ge 1$ when $k \leq n/q$, so $k' \ge n/q$.
    For any positive integer $t$, we have that
    \[
        g(k'+t) \leq \left(\frac{n-(k'+t)+1}{(k'+1)q}\right)^tg(k').
    \]
     Set $t = \log^{10}n$. Then the first factor on the right-hand side can be bounded from above by  $c^{\log^{10}n}$ for some constant $c<1$.
    Since $\frac{g(k)}{g(k+1)} = q\cdot \frac{k+1}{n-k}$ for any $k$, our choice of $k'$ ensures that we may bound $g(k')$ above by a constant.
    We conclude that $g(k'+t) = o(1)$, so $X_{k'+t} = 0$ with high probability by Markov's inequality.
    In the other direction, observe that
    \[
    g(k'-t) \geq \left(\frac{k'-t+1}{n-k'+t}q\right)^tg(k').
    \]
    As $k' \ge n/q$, this quantity is exponentially large in $t$, it follows that $\mathbb{E}[X_{k'-t}]=e^{\Omega(\log^{10}n)}$.

    Now we set $k = k' - \log^{10}n$ and show that $X_k$ is asymptotically concentrated around its mean.
    To this end, it suffices to show that $\Var[X_k] = o(\mathbb E[X_k]^2)$ and to apply Chebyshev's inequality.
    For every $I\in \binom{V}{k}$, let $X_I$ be the random variable indicating whether or not $G[I]$ is a good subgraph.
    We write the variance of $X_k$ in terms of these indicator variables.
    \begin{equation}\label{1.1: variance bound}
    \Var[X_k] = \sum_{I\in \binom{V}{k}}\Var[X_I] + \sum_{I\neq J\in \binom{V}{k}}\Cov(X_I, X_J) \leq \E[X_k] + \sum_{I\neq J\in \binom{V}{k}}\Cov(X_I, X_J).
    \end{equation}
    As $\mathbb E[X_k]$ is clearly $o(\mathbb E[X_k]^2)$, it suffices to show that the same is true for the covariance sum.\smallskip

    If the sets $I$ and $J$ are disjoint, then their corresponding submatrices do not overlap, so the variables $X_I$ and $X_J$ are independent and $\Cov(X_I, X_J) = 0$.
    On the other hand, if $|I\cap J|$ is large, say greater than $k-\log^2 n$, then the submatrices corresponding to $I$ and $J$ share many entries and $\Cov(X_I, X_J)$ may be large.
    Let $\mcal{B}$ be the collection of all pairs of such sets with large overlap,
    \begin{equation}\label{1.1: big overlap}
    \mcal{B} = \left\{(I, J): I,J\in \binom{V}{k},\ I\neq J,\ |I\cap J| \geq k-\log^2 n\right\}.
    \end{equation}
    Notice that $|\mcal{B}| \leq \binom{n}{k}\binom{k}{\log^2 n}\binom{n}{\log^2 n}\log^2 n$, so we can bound $\mcal{B}$'s contribution to the covariance. For any $I'\in \binom{V}{k}$ we have
    \begin{equation}\label{1.1: big overlap bound}
    \begin{split}
        \sum_{(I,J)\in \mcal{B}}\Cov(I, J) &\leq \binom{n}{k}\binom{k}{\log^2 n}\binom{n}{\log^2n}\log^2 n\cdot \max_{(I,J)\in \mcal{B}} \E[X_IX_J]\\
        &\leq \binom{n}{k}\binom{k}{\log^2 n}\binom{n}{\log^2 n}\log^2 n\cdot \E[X_{I'}]\\
        &= \binom{k}{\log^2 n}\binom{n}{\log^2 n}\log^2 n\cdot \E[X_k]\\
        &= o(\E[X_k]^2).
    \end{split}
    \end{equation}
    
    Now we stratify $\mcal{B}^C$ (which is the complement of $\mcal{B})$ according to the size of the overlap.
    Suppose that $(I, J)\in \mcal{B}^C$ and  $|I\cap J| = d$. The corresponding covariance term is
    \[
    \Cov(X_I, X_J) = \mathbb P[A_I\boldsymbol{1}_k  \equiv r\boldsymbol{1}_k,\ A_J\boldsymbol{1}_k \equiv  r\boldsymbol{1}_k] - \mathbb P[A_I\boldsymbol{1}_k \equiv r\boldsymbol{1}_k]\cdot \mathbb P[A_J\boldsymbol{1}_k \equiv r\boldsymbol{1}_k],
    \]
    where $A_I$ is the principal submatrix of $A(G)$ obtained by removing row and column $i$ for all $i\notin I$.
    Suppose that $A_I$ and $A_J$ are both adjacency matrices of good subgraphs.
    Without loss of generality, $A_I$ and $A_J$ overlap on the upper left $d\times d$ corner of $A_J$, which we call $U$.
    In order for $A_J\boldsymbol{1}_k \equiv r\boldsymbol{1}_k$, the upper right $d\times (k-d)$ submatrix of $A_J$, $M$, must complete the first $d$ rows of $A_J$ so that they each sum to $r \pmod{q}$.
    This event, which we call $E_{asym}$, occurs precisely when
    \[
    M\boldsymbol{1}_{k-d} = r\boldsymbol{1}_d - U\boldsymbol{1}_d
    \]
    By the symmetry of $A_J$, the submatrix $M$ determines the lower left $(k-d)\times d$ submatrix of $A_J$.
    The lower-right $(k-d)\times (k-d)$ submatrix, $L$, must then complete the bottom $k-d$ rows so that they each sum to $r\pmod{q}$.
    This event, which we call $E_{sym}$, occurs when
    \[
    L\boldsymbol{1}_{k-d} = r\boldsymbol{1}_d - M^T\boldsymbol{1}_d 
    \]
    This is illustrated in Figure \ref{fig:matrix}.\smallskip

    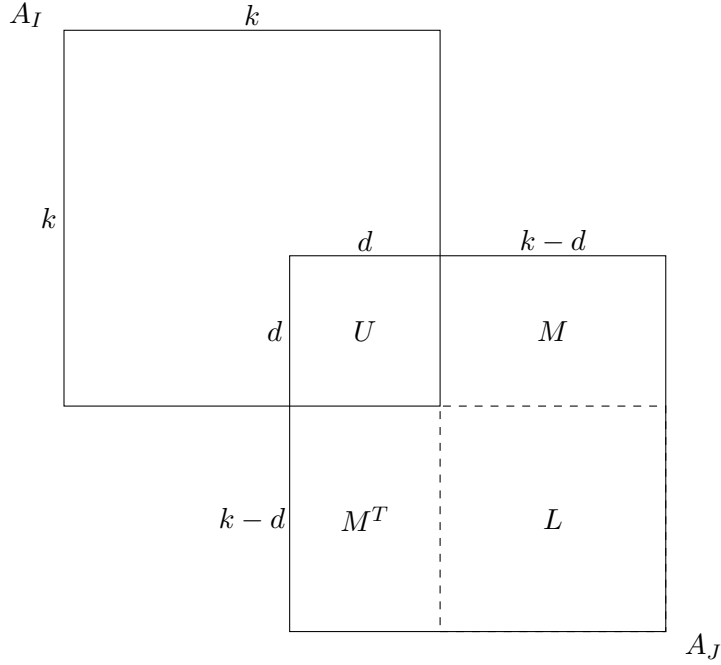
\begin{figure}[ht]
    \centering
    \begin{tikzpicture}
        \draw (0, 0) rectangle (5, 5);
        \draw (3, 2) rectangle (8, -3);
        \draw[dashed] (5, 0) rectangle (8, -3);

        \node[] at (-0.5, 5.2) {$A_I$};
        \node[] at (8.5, -3.2) {$A_J$};
        \node[] at (2.5, 5.2) {$k$};
        \node[] at (-0.2, 2.5) {$k$};
        \node[] at (4, 2.2) {$d$};
        \node[] at (4, 1) {$U$};
        \node[] at (6.5, 1) {$M$};
        \node[] at (2.8, 1) {$d$};
        \node[] at (6.5, 2.2) {$k-d$};
        \node[] at (2.5, -1.5) {$k-d$};
        \node[] at (4, -1.5) {$M^T$};
        \node[] at (6.5, -1.5) {$L$};
    \end{tikzpicture}
    \caption{Overlapping submatrices, $A_I$ and $A_J$, of $A(G)$.}
    \label{fig:matrix}
    \end{figure}

    We compute the covariance $\Cov(X_I, X_J)$ by conditioning on the event ``$A_I\boldsymbol{1}_k = r\boldsymbol{1}_k$'' and applying Corollary \ref{lem: non-symmetric} and Lemma \ref{lem: symmetric} as follows:
    \begin{equation}\label{covariance1}
    \begin{split}
        \Cov(X_I, X_J) &= \mathbb P[A_I\boldsymbol{1}_k  \equiv r\boldsymbol{1}_k,\ A_J\boldsymbol{1}_k \equiv  r\boldsymbol{1}_k] - \mathbb P[A_I\boldsymbol{1}_k \equiv r\boldsymbol{1}_k]\cdot \mathbb P[A_J\boldsymbol{1}_k \equiv r\boldsymbol{1}_k]\\
        &= \mathbb P[A_I\boldsymbol{1}_k \equiv r\boldsymbol{1}_k]\cdot \mathbb P[A_J\boldsymbol{1}_k \equiv r\boldsymbol{1}_k\mid A_I\boldsymbol{1}_k \equiv r\boldsymbol{1}_k] - \mathbb P[A_I\boldsymbol{1}_k \equiv r\boldsymbol{1}_k]^2\\
        &= \mathbb P[A_I\boldsymbol{1}_k \equiv r\boldsymbol{1}_k]\cdot \mathbb P[E_{asym}\mid A_I\boldsymbol{1}_k \equiv r\boldsymbol{1}_k]\cdot \mathbb P[E_{sym}\mid E_{asym},\ A_I\boldsymbol{1}_k \equiv r\boldsymbol{1}_k] - \mathbb P[A_I\boldsymbol{1}_k \equiv r\boldsymbol{1}_k]^2\\
    \end{split}
    \end{equation}
    By Lemma \ref{lem: symmetric}, $\Pr[A_I\boldsymbol 1_k \equiv r\boldsymbol 1_k] = \frac{1}{q^k}(1+e^{-\Omega(n)})$ when $q$ is odd and twice this when $q$ and $k$ are both even.
    Now for any value $r\boldsymbol 1_d - U\boldsymbol 1_d$ takes, $M\boldsymbol 1_{k-d}$ takes this same value with probability $1/q^d$ plus some uniformly small error by Corollary \ref{lem: non-symmetric}, so
    \[
        \mathbb P[E_{asym}\mid A_I\boldsymbol{1}_k \equiv r\boldsymbol{1}_k] = \frac{1}{q^d}\left(1 + e^{-\Omega(k-d)}\right)^d = \frac{1}{q^d}\left(1+e^{-\Omega(\log^2n)}\right)^d.
    \]
    Similarly, when $q$ is odd, $L\boldsymbol{1}_{k-d}$ takes all values in $\Z_q^{k-d}$ with probability $1/q^{k-d}$ plus some uniformly bounded error by Lemma \ref{lem: symmetric}, so
    \[
        \mathbb P[E_{sym}\mid E_{asym},\ A_I\boldsymbol{1}_k \equiv r\boldsymbol{1}_k] = \frac{1}{q^{k-d}}\left(1 + e^{-\Omega(k-d)}\right) = \frac{1}{q^{k-d}}\left(1 + e^{-\Omega(\log^2n)}\right).
    \]
    When $q$ is even, however, $L\boldsymbol{1}_{k-d}$ takes only the values in $\Z_q^{k-d}$ whose entrywise sums are even.
    In this case, we split the event $E_{sym}$ according to the parity of the entrywise sum of $r\boldsymbol 1_d - M^T\boldsymbol 1_d$.
    When this sum is even, $E_{sym}$ occurs with twice the above probability by Lemma \ref{lem: symmetric}, and it occurs with probability zero when the sum is odd.
    We then have (again, for odd $q$ -- make the aforementioned changes for even $q$)
    \begin{equation}\label{covariance2}
    \begin{split}
        \Cov(X_I, X_J)&= \frac{1}{q^k}\left(1+e^{-\Omega(n)}\right)\cdot \frac{1}{q^d}\left(1+e^{-\Omega(\log^2n)}\right)^d\cdot \frac{1}{q^{k-d}}\left(1+e^{-\Omega(\log^2n)}\right) - \frac{1}{q^{2k}}\left(1+e^{-\Omega(n)}\right)^2.
    \end{split}
    \end{equation}
    Since $k = \Theta(n)$ and $d \leq k - \log^2 n$, the above quantity is $\frac{1}{q^{2k}} \cdot o(1)$, from which we conclude
    \begin{equation}\label{1.1: small overlap bound}
    \sum_{(I,J)\in \mcal{B}^C}\Cov(X_I, X_J) \leq \binom{n}{k}^2 \cdot \frac{1}{q^{2k}}\cdot o(1) = o(\E[X_k]^2).
    \end{equation}
    
    Combining (\ref{1.1: variance bound}), (\ref{1.1: big overlap bound}) and (\ref{1.1: small overlap bound}), we see that $\Var[X_k] = o(\E[X_k])^2$, so with high probability $G$ contains a good subgraph of order $k$.
\end{proof}

\subsection{Proof of Theorem \ref{thm: one part distrubution}}

We slightly modify the second moment argument that we used to prove Theorem \ref{thm: one part}, again assuming that $q$ is odd for convenience.

\bigskip

\begin{proof}[Proof of Theorem \ref{thm: one part distrubution}]


    
    Given an integer $k$ and any $k_0, \ldots, k_{q-1}$ with $k_i\in \{\lceil \alpha_ik\rceil, \lfloor \alpha_ik\rfloor\}$ and $k_0 + \cdots + k_{q-1} = k$, let $S_{k_0, \ldots, k_{q-1}}$ be the set of vectors $v\in \Z_q^k$ having $k_i$ many entries congruent to $i \pmod{q}$ for all $i$.
    The size of $S_{k_0, \ldots, k_{q-1}}$ is given by the multinomial coefficient
    \[
        |S_{k_0, \ldots, k_{q-1}}| = \binom{k}{k_0, \ldots, k_{q-1}}.
    \]
    We use Lemma \ref{lem: symmetric} to estimate the probability that $M\boldsymbol{1}_k$ lies in $S_{k_0, \ldots, k_{q-1}}$, where $M$ is any $k\times k$ principal submatrix of $A(G)$:
    \[
    \mathbb{P}[M\boldsymbol{1}_k \in S_{k_0, \ldots, k_{q-1}}] = |S_{k_0, \ldots, k_{q-1}}|\left(\frac{1}{q^k} + \frac{o(1)}{q^k}\right),
    \]
    for odd $q$ and twice this when $q$ is even and $k$ and $\boldsymbol \alpha$ give rise to a valid degree sequence modulo $q$.
    The remainder of the proof is nearly identical to that of Theorem \ref{thm: one part}.
    If $X_k$ is the number of subgraphs of $G$ of order $k$ satisfying the conclusion of the theorem for this choice of $k_i$'s (``good'' subgraphs for short), then
    \[
    \mathbb{E}[X_k] = |S_{k_0, \ldots, k_{q-1}}|\binom{n}{k}\left(\frac{1}{q^k} + \frac{o(1)}{q^k}\right).
    \]
    Let $k'$ be the largest integer such that $|S_{k_0, \ldots, k_{q-1}}|\binom{n}{k}\frac{1}{q^{k}}\ge 1$ for all choices of the $k_i$ as described above, and let $\epsilon>0$ be arbitrarily small.
    A routine calculation  (see Appendix) shows that
    \begin{equation} \label{calc for appen}
        \mathbb{E}[X_{k'+\epsilon n}]=o(1) \textrm{ and } \mathbb{E}[X_{k'-\epsilon n}]=2^{f(\epsilon)n} \textrm{ for some function }f.
    \end{equation}
      In particular, by Markov's inequality we obtain that whp $X_{k'+\epsilon n}=0$. Now, let $k:=k'-\epsilon n$, and as in the proof of the previous theorem, we will show that $\Var[X_{k}] = o(\mathbb E[X_{k}]^2)$.
    As per (\ref{1.1: variance bound}), it suffices to show that
    \[
    \sum_{I\neq J\in \binom{[n]}{k}}\Cov(X_I, X_J) = o(\E[X_k]^2),
    \]
    where $X_I$ and $X_J$ are as they were in the proof of the previous theorem, except now a good subgraph is one whose distribution of degrees is given by $\boldsymbol \alpha$.
    Letting $\mcal{B}$ be as in (\ref{1.1: big overlap}), calculation (\ref{1.1: big overlap bound}) carries over exactly.
    Furthermore, calculations (\ref{covariance1}) and (\ref{covariance2}) carry over almost line for line since for any $k$,
    \[
    \begin{split}
    \mathbb{P}[A_I\boldsymbol{1}_k \in S_{k_0, \ldots, k_{q-1}},\ A_J\boldsymbol{1}_k\in S_{k_0, \ldots, k_{q-1}}] &= \sum_{\boldsymbol u,\boldsymbol v\in S_{k_0, \ldots, k_{q-1}}}\mathbb{P}[A_I\boldsymbol{1}_k = \boldsymbol u,\ A_J\boldsymbol{1}_k = \boldsymbol v]\\
    & \leq |S_{k_0, \ldots, k_{q-1}}|^2 \cdot \max_{\boldsymbol u, \boldsymbol v\in S_{k_0, \ldots, k_{q-1}}}\mathbb{P}[A_I\boldsymbol{1}_k = \boldsymbol u,\ A_J\boldsymbol{1}_k = \boldsymbol v].
    \end{split}
    \]
    In particular, $\sum_{(I,J)\in \mcal{B}^C}\Cov(X_I, X_J)$ is at most a factor of $|S_{k_0, \ldots, k_{q-1}}|^2$ larger than it was in the proof of Theorem \ref{thm: one part}.
    Since $\E[X_k]$ is exactly a factor of $|S_{k_0, \ldots, k_{q-1}}|$ larger than it was in Theorem \ref{thm: one part}, we still have that $\Var[X_k] = o(\E[X_k]^2)$ and the second moment argument goes through.
\end{proof}

\section{Packing}

We employ another second moment argument to prove Theorem \ref{thm: packing}.
Before getting into the proof we introduce some terminology. Suppose $V(G)=V_1\cup \ldots \cup V_t$ is a partition of the vertex set of a graph $G$, and $n_i\coloneqq|V_i|$ for all $i$.
Such a partition is \emph{balanced} if $n_i\in \{\lfloor n/t\rfloor,\lceil n/t\rceil\}$ for all $i$. For any fixed integers $q \geq 2$ and $0\leq r < q$, we call a partition $V(G) = V_1 \cup \ldots \cup V_t$ an $(r,q)$-partition if each vertex in $G[V_i]$ has degree congruent to $r\pmod q$ for all $i \leq t$.

Now let us assume for simplicity that $q$ is odd (the case of $q=2$ was settled by Scott in \cite{scott2001odd}; for other even $q$, whenever we use Lemma \ref{lem: symmetric} one needs to multiply the estimate by $2$) and let $t=q+1$.
Let $X$ be the number of balanced $(r, q)$-partitions of the vertex set of $G:= G_{n,1/2}$ into $t$ parts $V_1, \ldots, V_t$.
By Lemma \ref{lem: symmetric}, for every $1\leq i\leq t$, the probability that all the vertices in the induced subgraph $G[V_i]$ have degree $r\pmod{q}$ is $\frac{1}{q^{n_i}}(1+o(1))$. We then have
\begin{eqnarray*}
    \E[X] &=& \binom{n}{n_1, \ldots, n_t}\cdot \frac{(1 + o(1))^t}{q^n}\\
    &=&\frac{n!}{\prod_{i=1}^tn_i!}\cdot \frac{1+o(1)}{q^n}\\
    &\ge& n^{-O(1)}\frac{n^n}{(n/t)^n}\cdot \frac{1}{q^n}\\
    &=&n^{-O(1)}\left(\frac{q+1}{q}\right)^n=e^{\Theta(n/q)},
\end{eqnarray*}
where the second to last equality holds by Stirling's approximation.

In particular we have that the expected number of $(r, q)$-partitions with $t$ parts is exponentially large in $n$ (assuming that $q$ is fixed).
Note that taking $t = q+1 > q$ is critical in achieving this exponential bound, and for smaller values of $t$ the expectation has a smaller order of magnitude.
Indeed, as Balister, Powierski, Scott and Tan show in \cite{scott2021}, the number of balanced $(r,q)$-partitions of $G_{n, 1/2}$ with $q$ parts is distributed like a Poisson random variable (for $q>2$ at least; the $q=2$ case is more nuanced, but there is still no ``with high probability'' statement).

Now we need to show that $\Var[X] = o(\E[X]^2)$ and then apply Chebyshev's inequality.
To this end, let $\mcal{P}_t$ be the set of all balanced partitions of $V(G)$ into $t$ parts.
Specifically, an element $\mathcal U$ of $\mathcal P_t$ is a collection of $t$ subsets of $V(G)$, $U_1, \ldots, U_t$, that form a balanced partition.
We also let $X_{\mathcal U}$ denote the random variable indicating whether or not the partition $\mathcal U \in \mathcal P_t$ is an $(r,q)$-partition.
Using this notation, we write $X = \sum_{\mathcal U\in \mathcal P_t} X_{\mathcal U}$ and our goal becomes estimating $\sum_{\mathcal U, \mathcal V\in \mathcal P_t}\Cov(X_{\mathcal U}, X_{\mathcal V})$.

We will achieve this goal by defining a collection, $\mathcal T_t$, of pairs of partitions $(\mathcal U, \mathcal V)$ that behave more or less independently of one another, i.e. the probability that both are simultaneously $(r,q)$-partitions is around $q^{-2n}$ (and hence their covariance is small).
We will show that this independent behavior is typical, in the sense that most pairs behave this way.
Then we split the sum to be estimated,
\[
    \sum_{\mcal{U}, \mcal{V}\in \mcal{P}_t}\Cov(X_{\mcal{U}}, X_{\mcal{V}}) = \sum_{(\mcal{U}, \mcal{V})\in \mcal{T}_t}\Cov(X_{\mcal{U}}, X_{\mcal{V}}) + \sum_{(\mcal{U}, \mcal{V})\in \mcal{T}_t^C}\Cov(X_{\mcal{U}}, X_{\mcal{V}}).
\]
The first sum is small compared to $\E[X]^2$ because $\Cov(X_{\mathcal U}, X_{ \mathcal V})$ is small for typical pairs $(\mathcal U, \mathcal V)$, which make up the majority of all pairs.
While $\Cov(X_{\mathcal U}, X_{\mathcal V})$ might be larger for atypical pairs, a negligible proportion of all pairs behave this way.

We expect two partitions $\mathcal U, \mathcal V$ to behave (nearly) independently if their parts are all ``significantly different'' from each other.
This intuition motivates the following definition.

\begin{definition}
A pair $(\mathcal U, \mathcal V)\in \mathcal P_t^2$ is called \emph{typical} if for all $i$ and $j$ we have $|U_i\cap V_j| \leq n/3t$.
We call all other pairs \emph{atypical}.
The set $\mathcal T_t$ consists exactly of all typical pairs in $\mathcal P_t^2$.
\end{definition}
Note that in any typical pair of partitions $(\mathcal U, \mathcal V)$, we have that each $U_i$ intersects at least three parts of $\mathcal V$ in at least, say, $\log^2n$ vertices (and the same is true if we reverse the roles of $\mathcal U$ and $\mathcal V$).
Now we show that this definition ensures the desired behavior.

\begin{claim}\label{typical cov}
    If $(\mathcal{U}, \mathcal{V})\in \mathcal T_t$, then $\Cov(X_{\mathcal U}, X_{\mathcal{V}}) = o(1)/q^{2n}$.
\end{claim}
\begin{proof}
    Let $(\mathcal{U}, \mathcal{V})\in \mathcal T_t$ and write $\mathcal{U}=\{U_1, \ldots, U_t\}$ and $\mathcal{V}= \{V_1, \ldots V_t\}$.
    We reveal portions of the subgraphs $G[U_i]$ in stages.
    In the first stage, depicted in Figure \ref{fig: packing} (a), reveal the edges in $G[U_i\cap V_j]$ for all $1\leq i,j\leq t$.
    If we look at the adjacency matrix $A(G[U_i])$ for some $i$, then we have revealed a sequence of block submatrices along its diagonal.
    As $\mathcal U$ and $\mathcal V$ are typical partitions, each $U_i$ intersects at least three parts of $\mathcal V$ in at least $\log^2n$ vertices.
    This allows us to choose, for each $i$, an index $j_i$ such that $|U_i\cap V_{j_i}|\geq \log^2n$ (note that the $j_i$'s are not necessarily distinct indices).
    We arrange the vertices in $G$ so that, for each $i$, the induced subgraph $G[U_i\cap V_{j_i}]$ corresponds to the bottom-right corner of the matrix $A(G[U_i])$.
   
    In the second stage, reveal, for each $i$, the remaining edges in $G[U_i\setminus V_{j_i}]$, i.e. those edges in $G[U_i]$ not incident to any vertex in $V_{j_i}$.
    This reveals the upper-left corner of $A(G[U_i])$, as shown in Figure \ref{fig: packing} (b).
    The still unrevealed edges in $G[U_i]$ correspond to those that cross between $V_{j_i}$ and $U_i\setminus V_{j_i}$.
    Call the portion of $A(G[U_i])$ corresponding to these vertices, the upper-right corner, $M_i$ and let its size be $s_i \times t_i$.
    In order for $G[U_i]$ to have all degrees congruent to $r\pmod q$, we must have
    \begin{equation}\label{necessary}
        M_i\boldsymbol 1_{t_i} \equiv r\cdot \boldsymbol{1}_{s_i} - A(G[U_i\setminus V_{j_i}])\boldsymbol{1}_{s_i},\qquad\text{and}\qquad \boldsymbol{1}_{s_i}^TM_i \equiv (r\boldsymbol{1}_{t_i}-A(G[U_i\cap V_{j_i}])\boldsymbol{1}_{t_i})^T,
    \end{equation}
    which gives us the following necessary condition on $A(G[U_i\cap V_{j_i}])$ and $A(G[U_i\setminus V_{j_i}])$ after the second stage of revealing entries:
    \begin{equation}\label{compatibility}
        r\cdot |U_i\setminus V_{j_i}| - \boldsymbol{1}_{s_i}^TA(G[U_i\setminus V_{j_i}])\boldsymbol{1}_{s_i} \equiv r\cdot |U_i\cap V_{j_i}| - \boldsymbol{1}_{t_i}^TA(G[U_i\cap V_{j_i}])\boldsymbol{1}_{t_i}.
    \end{equation}
    We say that the subgraphs $G[U_i\cap V_{j_i}]$ and $G[U_i\setminus V_{j_i}]$ are \emph{compatible} in $G[U_i]$ when congruence (\ref{compatibility}) holds, and Lemma \ref{lem: both ways}, applied to $M_i$ and (\ref{necessary}), gives the probability that $G[U_i]$ has all degrees congruent to $r\pmod q$ given this compatibility.

    Note that the compatibility condition (\ref{compatibility}) depends only on the numbers of edges in $G[U_i\cap V_{j_i}]$ and $G[U_i\setminus V_{j_i}]$ modulo $q$.
    Since $(\mathcal U, \mathcal V)$ is typical, $U_i$ intersects at least two other parts of $\mathcal V$ other than $V_{j_i}$ in at least $\log^2n$ vertices, so the number of entries revealed in the second stage is $\Omega(\log^4n)$.
    By Lemma \ref{lem: one row}, the probability that $G[U_i\cap V_{j_i}]$ and $G[U_i\setminus V_{j_i}]$ are compatible in $G[U_i]$ is $\frac{1+o(1)}{q}$.
    
    In the third and final stage, shown in Figure \ref{fig: packing} (c), we reveal the remaining edges in $G[U_i]$, which are those that cross between $U_i\cap V_{j_i}$ and $U_i\setminus V_{j_i}$ and correspond to the matrix $M_i$.
    We condition on $G[U_i\cap V_{j_i}]$ and $G[U_i\setminus V_{j_i}]$ being compatible in $G[U_i]$, and apply Lemma \ref{lem: both ways} to conclude that this revelation produces a good subgraph with probability $\frac{1+o(1)}{q^{|U_i| -1 }}$.
    Each revealing step in $U_i$ is independent of the revealing steps in $U_{i'}$ for distinct $i,i'$, so each $G[U_i]$ is good independently with probability $\frac{1+o(1)}{q^{|U_i|}}$.
    
    \begin{figure}
    \centering
    \subfigure[{Stage 1: reveal $G[U_i \cap V_j]$ for all $i$ and $j$.}]{
        \scalebox{0.55}{
            \begin{tikzpicture}
                \node[] at (6.5, 14.5) {\huge $U_i$};
                \draw (0, 0)  rectangle (14, 14);
                \draw[fill=gray] (0, 14) rectangle (1, 13);
                \draw[fill=gray] (1, 13) rectangle (2, 12);
        
                \node[] at (2.5, 11.5) {\huge $\ddots$};
        
                \draw[fill=gray] (3, 11) rectangle (4, 10);
                \draw[fill=gray] (4, 10) rectangle (7, 7);
                \draw[fill=gray] (7, 7)  rectangle (10, 4);
        
                \node[] at (10.5, 3.5) {\huge $\ddots$};
        
                \draw[fill=gray] (11, 3) rectangle (14, 0);
        
                \node[] at (12.5, 1.5) {\huge $U_i \cap V_{j_i}$};
                
            \end{tikzpicture}
        }
    }
    \subfigure[{Stage 2: reveal the rest of $G[U_i \setminus V_{j_i}]$.}]{
        \scalebox{0.55}{
            \begin{tikzpicture}
                \node[] at (6.5, 14.5) {\huge $U_i$};
                \draw (0, 0)  rectangle (14, 14);
                \draw[fill=gray] (0, 14) rectangle (11, 3);
                \draw[fill=gray!30] (0, 14) rectangle (1, 13);
                \draw[fill=gray!30] (1, 13) rectangle (2, 12);
        
                \node[] at (2.5, 11.5) {\huge $\ddots$};
        
                \draw[fill=gray!30] (3, 11) rectangle (4, 10);
                \draw[fill=gray!30] (4, 10) rectangle (7, 7);
                \draw[fill=gray!30] (7, 7)  rectangle (10, 4);
        
                \node[] at (10.5, 3.5) {\huge $\ddots$};
        
                \draw[fill=gray!30] (11, 3) rectangle (14, 0);
        
                \node[] at (12.5, 1.5) {\huge $U_i \cap V_{j_i}$};
                
            \end{tikzpicture}
        }
    }

    \subfigure[{Stage 3: reveal the rest of $G[U_i]$.}]{
        \scalebox{0.6}{
            \begin{tikzpicture}
                \node[] at (6.5, 14.5) {\huge $U_i$};
                \draw (0, 0)  rectangle (14, 14);
                \draw[fill=gray!30] (0, 14) rectangle (11, 3);
                \draw[fill=gray!30] (0, 14) rectangle (1, 13);
                \draw[fill=gray!30] (1, 13) rectangle (2, 12);
        
                \node[] at (2.5, 11.5) {\huge $\ddots$};
        
                \draw[fill=gray!30] (3, 11) rectangle (4, 10);
                \draw[fill=gray!30] (4, 10) rectangle (7, 7);
                \draw[fill=gray!30] (7, 7)  rectangle (10, 4);
        
                \node[] at (10.5, 3.5) {\huge $\ddots$};
        
                \draw[fill=gray!30] (11, 3) rectangle (14, 0);
        
                \node[] at (12.5, 1.5) {\huge $U_i \cap V_{j_i}$};
                
                \draw[fill=gray] (0, 0) rectangle (11, 3);
                \draw[fill=gray] (11, 3) rectangle (14, 14);
            \end{tikzpicture}
        }
    }
    \caption{The adjacency matrix of $G[U_i]$ when performing the three revealing stages.}
    \label{fig: packing}
    \end{figure}
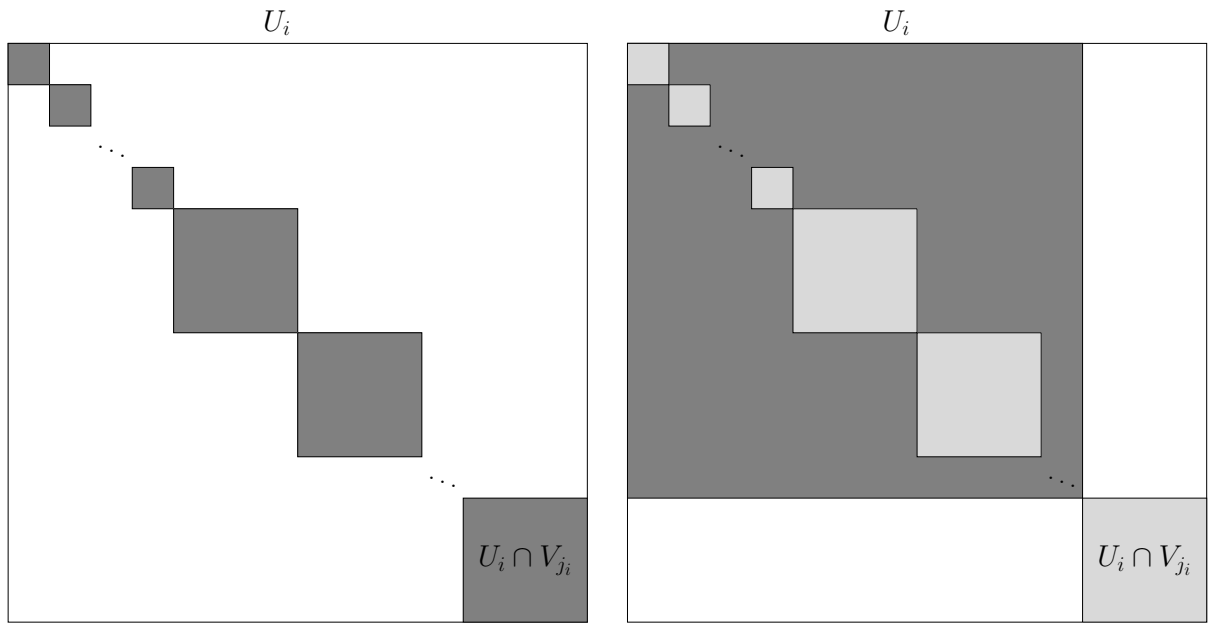
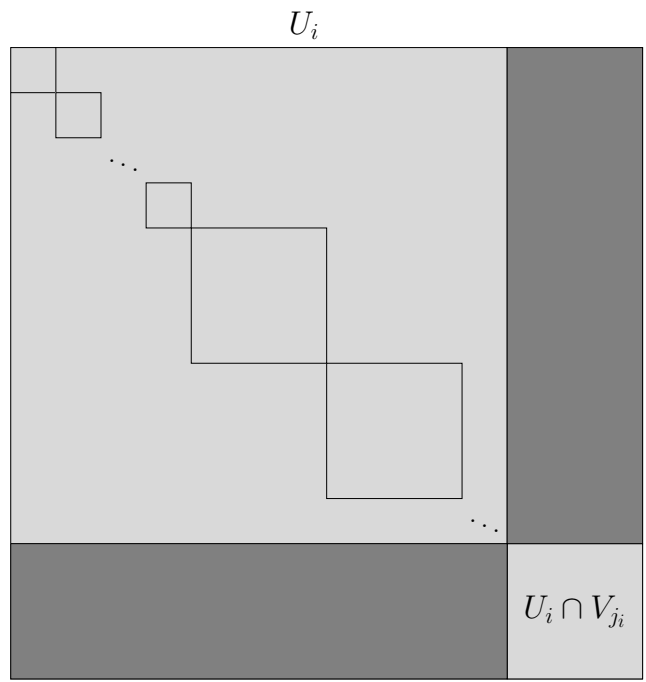
    
    We repeat the three revealing stages for partition $\mathcal V$.
    Since $\mathcal V$ is typical, each $V_i$ intersects at least three parts of $\mathcal U$ in at least $\log^2n$ vertices and for each $i$, we choose an index $j_i$ such that $|V_i\cap U_{j_i}| \geq \log^2n$.
    The first stage would have us reveal $G[V_i \cap U_j]$ for all $i$ and $j$, but we have already revealed these subgraphs.
    In the second stage we reveal, for each $i$, the edges in $G[V_i]$ not incident to any vertex in $U_{j_i}$.
    Each such edge crosses between $G[V_i \cap U_j]$ and $G[V_i\cap U_{j'}]$ for some $j\neq j'$.
    Prior to this step we had only revealed the edges in the induced subgraphs $G[U_i]$ for all $i$, so no such crossing was revealed and we have enough randomness to apply Lemma \ref{lem: one row} for the compatibility condition.
    Likewise, in stage three we reveal the edges that cross between $G[V_i \cap U_{j_i}]$ and $G[V_i \setminus U_{j_i}]$, all of which were previously unrevealed by the same argument.
    We conclude that $\mathcal U$ and $\mathcal V$ are both $(r,q)$-partitions with probability $\frac{1+o(1)}{q^{2n}}$.
\end{proof}

Next, we show that the atypical pairs make up an exponentially small fraction of $\mathcal P_t^2$.
For ease of notation, we let $M(n,t)$ denote the number of balanced partitions of $V(G)$ into $t$ sets, i.e. $M(n,t)\coloneqq \binom{n}{n_1, \ldots, n_t}$, where $n_i\in \{\lfloor n/t\rfloor, \lceil n/t\rceil\}$.
Note that this quantity is indeed well defined since the $n_i$ are unique up to relabelling.

\begin{claim}\label{atyp size}
    The number of atypical pairs of partitions, $|\mathcal T_t^c|$, is at most $M(n,t)^2\cdot e^{-\epsilon n}$ for some constant $\epsilon$ that depends only on $q$.
\end{claim}

\begin{proof}
    Fix a partition $\mathcal U$ and sample $\mathcal V$ uniformly at random from $\mathcal P_t$.
    For any $i$ and $j$, the quantity $|U_i\cap V_j|$ is hypergeometrically distributed with mean $n/t^2$.
    Since $q \geq 3$, a simple application of Remark \ref{CheHyper} shows that the probability that there exist $U_i$ and $V_j$ such that $|U_i\cap V_j| \geq \frac{|U_i|}{3}$ is at most $t^2e^{-\Theta(n)}$, where the implicit constant depends only on $q$.
    Since there are at most $M(n,t)^2$ pairs of balanced partitions, the claim follows.
\end{proof}

Now split the atypical partitions into $\mathcal T_t^c = \mathcal B_1\cup \mathcal B_2$, where $\mathcal B_1$ is the set of atypical pairs where for for all $i$ and $j$ we have $|U_i\setminus V_j|, |V_j\setminus U_i| \geq \log^2n$.
In other words, $\mathcal B_1$ represents the pairs of atypical partitions where no symmetric difference $U_i \triangle V_j$ is too small.
Let us estimate how much these contribute to the covariance.

\begin{claim}\label{claim: B1}
    If $(\mathcal U, \mathcal V)\in \mathcal B_1$, then
    \[
    \mathbb P[\mathcal U \text{ and }\mathcal V\text{ are both }(r,q)\text{-partitions}] \leq (1+o(1))\cdot \frac{1}{q^{2n}}\cdot q^{2t}.
    \]
\end{claim}
\begin{proof}
    For every $i$ there exists some $j_i$ so that $|U_i \cap V_{j_i}| \geq |U_i| / t = n/t^2$ (here we are ignoring any rounding since it is inconsequential to the argument).
    Like in the proof of Claim \ref{typical cov}, we reveal the subgraphs $G[U_i]$ in stages.
    In the first stage, we reveal $G[U_i \cap V_j]$ for all $i$ and $j$.
    
    In stage two, we reveal the remaining entries in $G[U_i \setminus V_{j_i}]$ for each $i$.
    Unlike in the proof of Claim \ref{typical cov}, we might have too few unrevealed edges at the end of the first stage to use Lemma \ref{lem: one row} to estimate the probability that $G[U_i \cap V_{j_i}]$ and $G[U_i \setminus V_{j_i}]$ are compatible in $G[U_i]$ (e.g., if $G[U_i]$ intersects only two parts of $\mathcal V$).
    Instead, we simply bound this probability of compatibility from above by 1.
    
    In the third step, we reveal the rest of the edges in $G[U_i]$, those that cross between $G[U_i \cap V_{j_i}]$ and $G[U_i \setminus V_{j_i}]$.
    The probability that this produces a good subgraph is at most what it would be given that $G[U_i\cap V_{j_i}]$ and $G[U_i \setminus V_{j_i}]$ are compatible, $\frac{1+o(1)}{q^{|U_i|-1}}$ by Lemma \ref{lem: both ways} and the assumption that $|U_i \setminus V_{j_i}| \geq \log^2n$.
    By the same argument used in the proof of Claim \ref{typical cov}, the revealing steps within each $G[U_i]$ are independent of one another and we may repeat this argument on  $G[V_i]$'s as well.
\end{proof}

Now for the remaining atypical pairs.
Further split $\mathcal B_2$ into $\mathcal B_2 = \cup_{s=1}^t\mathcal B_{2,s}$, where $\mathcal B_{2,s}$ is the set of atypical pairs $(\mathcal U, \mathcal V)$ where the number of $i$ and $j$ such that $|U_i\setminus V_j|$ or $|V_j\setminus U_i|$ is at most $\log^2n$ is exactly $s$.
In other words, $\mathcal B_{2,s}$ is the set of atypical pairs that have $s$ parts (nearly) in common.
Let us estimate the probability that a pair of partitions in $\mathcal B_{2,s}$ are both simultaneously $(r,q)$-partitions.

\begin{claim}\label{claim: B2}
    For any pair $(\mcal{U}, \mcal{V})\in \mathcal B_{2,s}$, we have
    \begin{equation*}
        \mathbb P[\mcal{U}\text{ and }\mcal{V}\text{ are $(r, q)$-partitions}] \leq \frac{1+o(1)}{q^{2n}}\cdot q^{s\lceil n/t\rceil  + 2(t-s)}.
    \end{equation*}
\end{claim}
\begin{proof}
    Since for every $i$, $|U_i\setminus V_j|\leq \log^2n$ can hold for at most one value of $j$, we may, without loss of generality, assume that $|U_i\setminus V_i|\leq \log^2n$ for $i \leq s$, i.e., that the first $s$ parts of both $\mathcal U$ and $\mathcal V$ are nearly identical.
    We have by Lemma \ref{lem: symmetric} that
    \begin{align*}
    \mathbb P[G[U_i] \text{ and }G[V_i] \text{ are good subgraphs, }i\leq s] &\leq \mathbb P[G[U_i] \text{ is a good subgraph, }i\leq s]\\
    & \leq (1+o(1))q^{-\sum_{i\leq s}|U_i|}.
    \end{align*}
    Then we apply the argument from the proof of Claim \ref{claim: B1} to the remaining $t-s$ parts of both $\mathcal U$ and $\mathcal V$.
    Specifically, for $i > s$, parts $U_i$ and $V_i$ have probability around $1/q^{|U_i|-1}$ and $1/q^{|V_i|-1}$ of being good, respectively.
    In total, the desired probability is then
    \begin{align*}
        (1+o(1))q^{-\sum_{i\leq s}|U_i|}\cdot q^{-\sum_{i > s}|U_i|}\cdot q^{-\sum_{i>s}|V_i|}\cdot q^{2(t-s)} \leq (1+o(1))q^{-2n}\cdot q^{s\lceil n/t\rceil}\cdot q^{2(t-s)}.
    \end{align*}
    This completes the proof.
\end{proof}

Next we estimate the size of each $\mathcal B_{2,s}$.

\begin{claim}\label{claim: bads}
    For any $1\leq s\leq t$,
    \[
    |\mcal B_{2,s}| \leq M(n, t)^2 \cdot q^{-s\lfloor n/t\rfloor}\cdot o(1).
    \]
\end{claim}

\begin{proof}
    There are $M(n, t)$ ways to choose the first partition $\mcal{U}$ and at most $t^{2s}$ ways to choose disjoint pairs $(U_{i_\ell},V_{j_\ell})$ for which $U_{i_{\ell}}$ and $V_{j_{\ell}}$ will overlap significantly (i.e., in more than $n/t-\log^2n$ vertices).
    Observe that if $U_i$ and $V_j$ overlap significantly, the $V_j$ cannot overlap significantly with any other $U_{\ell}$. 
    
    Consider such a pair $(U_{i},V_{j})$.
    At most $\log^2n$ of the vertices in $V_j$ miss the vertices in $U_i$ and land in the remaining $n - |U_i|$ vertices of $G$.
    There are then at most
    \[
    \binom{\lceil n/t\rceil }{\log^2n}^s\cdot \binom{n}{\log^2n}^s \leq \binom{n}{\log^2n}^{2s}
    \]
    ways to choose the elements of the $s$ parts of $\mcal{V}$ that overlap significantly with $\mcal{U}$.
    There are at most $n-s\lfloor n/t\rfloor$ vertices and $t-s$ parts of $\mcal{V}$ left to fill and there are at most $M(n-s\lfloor n/t\rfloor, t-s)$ ways to do this. So far, we have that
    \begin{equation*}
    \begin{split}
    |\mcal{B}_{2,s}|&\leq M(n,t)\cdot t^{2s}\cdot \binom{n}{\log^2n}^{2s}\cdot M\left(n-s\lfloor n/t\rfloor, t-s\right) \\
    &\leq M(n,t)\cdot M\left(n-s\lfloor n/t\rfloor, t-s\right)\cdot (tn^{\log^2n})^{2s},
    \end{split}
    \end{equation*}
    for $n$ sufficiently large.
    For any fixed $s$, it suffices to show that
    \[
    M\left(n - s\lfloor n/t\rfloor, t-s\right) \leq M(n,t)\cdot q^{-sn/t}\cdot o\left(n^{-2s\log^2n}\right).
    \]
    
    By Stirling's approximation we have
    \[
    M(n, t) \sim \frac{t^n}{n^{(t-1)/2}}\cdot C(q),
    \]
    where $C(q)$ is some constant depending only on $q$.
    We bound $M(n - s\lfloor n/t \rfloor, t-s)$ as follows:
    \[
    \begin{split}
        M(n-s\lfloor n/t \rfloor, t-s) &\leq (t-s)^{n - s\lfloor n/t \rfloor}\\
        &\leq t^nt^{-s\lfloor n/t \rfloor}\\
        &\sim  \frac{1}{C(q)}\cdot M(n, t)\cdot (q+1)^{-s\lfloor n/t \rfloor}\cdot n^{t/2 - 1} \\
        &\leq \frac{1}{C(q)} \cdot M(n,t)q^{-s\lfloor n/t \rfloor}e^{-s\lfloor n/t \rfloor/q}\cdot n^{t/2-1} \\
        &= M(n,t)q^{-s\lfloor n/t \rfloor}o(n^{-2s\log^2n}).
    \end{split}
    \]
    This completes the proof. 
\end{proof}

By Claim \ref{typical cov}, in the typical case we have
\[
    \sum_{(\mcal{U}, \mcal{V})\in \mcal{T}_t}\Cov(X_{\mcal{U}}, X_{\mcal{V}}) \leq M(n, t)^2\cdot q^{-2n}\cdot o(1) = o(\E[X]^2).
\]
In the atypical case, we use Claim \ref{atyp size} to bound the size of $\mathcal B_1$ and Claim \ref{claim: B1} to estimate the covariance for each term here.
We similarly use Claims \ref{claim: bads} and \ref{claim: B2} to estimate the size of $\mathcal B_{2,s}$ and the covariance of each term here, respectively.
\[
\begin{split}
    \sum_{(\mcal{U}, \mcal{V})\in \mcal{T}_t^C}\Cov(X_{\mcal{U}}, X_{\mcal{V}}) & = \sum_{(\mathcal U, \mathcal V)\in \mathcal B_1}\Cov(X_{\mathcal U}, X_{\mathcal V}) + \sum_{s=1}^t\sum_{(\mathcal U, \mathcal V)\in \mathcal B_{2,s}}\Cov(X_{\mathcal U}, X_{\mathcal V})\\
    &\leq |\mathcal B_1|\cdot \frac{1+o(1)}{q^{2n}}(q^{2t}-1) + \sum_{s=1}^t|\mathcal B_{2,s}|\cdot \frac{1+o(1)}{q^{2n}}(q^{s\lceil n/t\rceil + 2(t-s)}-1)\\
    &\leq M(n,t)^2\cdot e^{-\epsilon n}\cdot \frac{1+o(1)}{q^{2n}}\cdot (q^{2t}-1) + M(n,t)^2\cdot o(1)\cdot \frac{1+o(1)}{q^{2n}}\cdot \sum_{s=1}^tq^{2(t-s)}\\
    &= o(\E[X]^2).
\end{split}
\]
Hence, $\Var[X] = o(\E[X]^2)$, so there exists an $(r,q)$-partition with high probability.

\bibliographystyle{abbrv}
\bibliography{partitioning}

\begin{thebibliography}{10}

\bibitem{as2015}
N.~Alon and J.~H. Spencer.
\newblock {\em The Probabilistic Method}.
\newblock Wiley, fourth edition, 2015.

\bibitem{scott2021}
P.~Balister, E.~Powierski, A.~Scott, and J.~Tan.
\newblock Counting partitions of {$G_{n,1/2}$} with degree congruence
  conditions.
\newblock {\em arXiv:2105.12612 [math.CO]}, 2021.

\bibitem{caro1994induced}
Y.~Caro.
\newblock On induced subgraphs with odd degrees.
\newblock {\em Discrete Mathematics}, 132(1-3):23--28, 1994.

\bibitem{info2004}
I.~Csiz\'ar and P.~C. Shields.
\newblock {\em Information Theory and Statistics: A Tutorial}.
\newblock Foundations and Trends in Communications and Information Theory. Now
  Publishers Inc, 2004.

\bibitem{ferber2020singularity}
A.~Ferber.
\newblock Singularity of random symmetric matrices -- simple proof.
\newblock {\em arXiv preprint arXiv:2006.07439 [math.PR]}, 2020.

\bibitem{FK}
A.~Ferber and M.~Krivelevich.
\newblock Every graph contains a linearly sized induced subgraph with all
  degrees odd.
\newblock {\em arXiv:2009.05495v3 [math.CO]}, 2021.

\bibitem{jlr2011}
S.~Janson, T.~{\L}uczak, and A.~Ruci\'nski.
\newblock {\em Random Graphs}.
\newblock John Wiley \& Sons, Inc., 2000.

\bibitem{lovasz1993}
L.~Lov\'asz.
\newblock {\em Combinatorial Problems and Exercises}.
\newblock AMS Chelsea Publishing, second edition, 1993.

\bibitem{scott1992large}
A.~Scott.
\newblock Large induced subgraphs with all degrees odd.
\newblock {\em Combinatorics, Probability and Computing}, 1(4):335--349, 1992.

\bibitem{scott2001odd}
A.~Scott.
\newblock On induced subgraphs with all degrees odd.
\newblock {\em Graphs and Combinatorics}, 17:539--553, 2001.

\bibitem{stein2011fourier}
E.~M. Stein and R.~Shakarchi.
\newblock {\em Fourier analysis: an introduction}, volume~1.
\newblock Princeton University Press, 2011.

\end{thebibliography}

\appendix

\section{Proof of Equation (\ref{calc for appen})}

The proof is based on bounding the multinomial coefficient in terms of the entropy function.
For any $\boldsymbol \alpha \in [0,1]^q$ with $\alpha_0 + \cdots + \alpha_{q-1} = 1$, its entropy (or really, the entropy of the corresponding random variable) is given by
\[
H(\boldsymbol \alpha) = -\sum_{i=0}^{q-1} \alpha_i \log_2\alpha_i,
\]
where we extend $x\mapsto x\log_2 x$ by continuity to the origin with $0\cdot \log_20 \coloneqq 0$.
If $p\in [0,1]$, then we define the binary entropy, $H(p) \coloneqq H((p, 1-p))$.
Now for $x\in[0,1]$ and $\boldsymbol \alpha$ as above, define
\[
h(x)=H(\boldsymbol \alpha)x+H(x)-\log_2q\cdot x.
\]
Some basic calculus reveals the following properties of $h(x)$:
\begin{enumerate}
\item $h'(x)=H(\boldsymbol \alpha)+\log_2\frac{1-x}{x}-\log q$ for $x\in (0,1)$;
\item $h''(x)=-\frac{1}{\ln 2\cdot x(1-x)}<0$ for $x\in (0,1)$;
\item $h(0)=0$, $h(1)=H(\boldsymbol \alpha)-\log_2q\le 0$, with $h(1) = 0$ if and only if $\alpha_i = 1/q$ for all $i$ (we call this the ``uniform case'');
\item for $0<x\le \frac{1}{q+1}$, $h'(x)\ge \log_2\frac{1-x}{x}-\log_2q\ge 0$, implying $h(x) \ge 0$ in this range;
\item $h'(x)=0$ in a unique point in $(0,1)$, this is a maximum point of $h(x)$ since $h''(x)<0$ in the interval. Hence $h(x)>0$ in this point.
\end{enumerate}

Define $x_0=\max\{x\in (0,1]: h(x)=0\}$.
This is well defined due to item (3) above.
In the non-uniform case we can see that $x_0<1$ and, due to the maximality of $x_0$, we have $h'(x_0)\le 0$ (in fact, $h'(x_0)<0$).
In this case, suppose $h'(x_0)=-a$ for $a>0$. Then $h'(a)<-a$ for $x>x_0$ by (2).
This implies that 
for $x=x_0+t$ (assuming $x_0+t<1$), $h(x)\le -at$ by the intermediate value theorem. Also, for $x=x_0-t$ and $t>0$ small enough, we have $h(x)=\Omega(at)$, due to the continuity of $h''(x)$.

Let now, for an integer $k\in [0,n]$:
\[
g(k)=\log_2  \left(\binom{k}{k_0,\ldots,k_{q-1}}\binom{n}{k}q^{-k}\right)\,.
\]
By the mean value theorem and the estimate (see, e.g. \cite{info2004}, Lemma 2.2)
\[
\log_2 \binom{k}{k_0, \ldots, k_{q-1}} = H\big((k_0/k, \ldots, k_{q-1}/k)\big)k - o(k),
\]
we have
\[
g(k)=H(\boldsymbol \alpha)k+H(k/n)n-\log_2 q\cdot k+o(n)\,,
\]
implying:
\begin{eqnarray*}
\frac{g(k)}{n}&=&H(\alpha_0,\ldots,\alpha_{q-1})\frac{k}{n}+H(k/n)-\log_2 q\cdot\frac{k}{n}+o(1)\\
&=& h(k/n)+o(1)\,.
\end{eqnarray*}
Assuming that $x_0<1$, choose an integer $k\in [0,n]$ satisfying $k\ge x_0n+\epsilon n$, then $g(k)/n=-\Theta(\epsilon)+o(1)$, implying that $2^{g(k)}=2^{-\Theta(n)}$. In the opposite direction (and in every case, including the uniform one) choose an integer $k\in [0,n]$ satisfying $k\le x_0n-\epsilon n$, for small enough $\epsilon>0$; we can conclude that $2^{g(k)}=2^{\Theta(n)}$.

\bigskip

\end{document}